\documentclass[10pt,a4paper]{amsart}
\usepackage[utf8]{inputenc}
\usepackage[english]{babel}
\usepackage{amsmath}
\usepackage{amsfonts}
\usepackage{amssymb}
\usepackage{amsthm}

\newtheorem{prop}{Proposition}
\newtheorem{lem}[prop]{Lemma}

\begin{document}
\title{Some points on Vervaat's transform of Brownian bridges and Brownian motion}
\author{Titus Lupu}
\email{titus.lupu@u-psud.fr}
\begin{abstract}
In a recent work J. Pitman and W. Tang defined the Vervaat's transform for a Brownian bridge with two different endpoints and for a Brownian motion between times $0$ and $1$. They proved some path decomposition properties for these Vervaat's transforms and raised open questions on their semi-martingale decomposition. In this paper we give an alternative proof for the path decomposition and answer the open questions.
\end{abstract}
\maketitle

\section{Introduction}
Given a continuous function $f$ on $[0,1]$ and $\tau(f)$ the first time it attains its minimum, Pitman and Tang define in \cite{PitmanTang2013Vervaat} the Vervaat's transform of $f$ as:
\begin{displaymath}
V(f)(t):=f(\tau(f)+t \mod 1)+f(1)1_{\lbrace t+\tau(f)\geq 1\rbrace}-f(\tau(f))
\end{displaymath}
They apply this transform to the following three cases:
\begin{itemize}
\item (i) a Brownian bridge $(B^{\lambda,br}_{t})_{0\leq t\leq 1}$ from $0$ to $\lambda$ with $\lambda<0$
\item (ii) a Brownian bridge $(B^{\lambda,br}_{t})_{0\leq t\leq 1}$ from $0$ to $\lambda$ with $\lambda>0$
\item (iii) a Brownian path $(B_{t})_{0\leq t\leq 1}$
\end{itemize}
For the case (i) they show the following path decomposition: given $Z$ a r.v. on $(0,1)$ with distribution:
\begin{displaymath}
\dfrac{\vert\lambda\vert}{\sqrt{2\pi t(1-t)^{3}}}\exp\left(-\dfrac{\lambda^{2}t}{2(1-t)}\right)\,1_{0<t<1}dt
\end{displaymath}
$V(B^{\lambda,br})$ can be decomposed as a positive Brownian excursion on $[0,Z]$ and a first passage bridge from $0$ to $\lambda$ on $[Z,1]$, independent conditionally on $Z$. The proof of the decomposition studies first the case of bridges of random walks then takes the weak limit. In case (ii) there is a similar decomposition derived from case (i) by time reversal: first a time-reversed first passage bridge from $\lambda$ to $0$ then a positive excursion above level $\lambda$. Further Pitman and Tang prove that in all three cases, the Vervaat transformed paths are semi-martingales, identify the decomposition in local martingale and finite variation process in case (i) and leave as an open question the decomposition in case (ii) and (iii).

In this paper we give an other proof of the path decomposition of $V(B^{\lambda,br})$ and identify the semi-martingale decomposition of Vervaat transformed paths in cases (ii) and (iii).

\section{Alternative proof of path decomposition of a Vervaat's transform of a Brownian bridge with arbitrary endpoints}

Let $\lambda<0$. Our proof of path decomposition of $V(B^{\lambda,br})$ relies on the decomposition of bridges at their minimum, similar to the decomposition at the maximum that appears in \cite{PitmanYor1996MaximumBridge}. Let $p_{T}(x,y)$ be the heat kernel
\begin{displaymath}
p_{T}(x,y)=\dfrac{1}{\sqrt{2\pi T}}\exp\left(-\dfrac{(y-x)^{2}}{2T}\right)
\end{displaymath}
Let $\mathbb{P}^{T}_{0,\lambda}$ the law of the Brownian bridge from $0$ to $\lambda$ of length $T$ and for $y<x$ let $\mathbb{P}_{x}^{T_{y}}$ the law of the Brownian path starting from $x$ until the first time it hits $y$. Given a distribution $Q$ on paths, $Q^{\wedge}$ will denote its image by time reversal. Given $Q$ and $Q'$ two distributions on paths, $Q\circ Q'$ will be the distribution obtained by concatenating two independent paths, one following the distribution $Q$ and the other the distribution $Q'$. According to corollary $3$ in \cite{PitmanYor1996MaximumBridge}:
\begin{displaymath}
\int_{0}^{+\infty}dT\,p_{T}(0,\lambda)\mathbb{P}^{T}_{0,\lambda}=2\int_{-\infty}^{\lambda}
dy\,\mathbb{P}_{0}^{T_{y}}\circ\mathbb{P}_{\lambda}^{T_{y}\wedge}
\end{displaymath}
$\mathbb{P}_{0}^{T_{y}}$ can be decomposed as:
\begin{displaymath}
\mathbb{P}_{0}^{T_{y}}=\mathbb{P}_{0}^{T_{y-\lambda}}\circ\mathbb{P}_{y-\lambda}^{T_{y}}
\end{displaymath}
Thus:
\begin{equation}
\label{Eq1}
\int_{0}^{+\infty}dT\,p_{T}(0,\lambda)\mathbb{P}^{T}_{0,\lambda}=2\int_{-\infty}^{\lambda}
dy\,\mathbb{P}_{0}^{T_{y-\lambda}}\circ
\mathbb{P}_{y-\lambda}^{T_{y}}\circ\mathbb{P}_{\lambda}^{T_{y}\wedge}
\end{equation}
We extend the definition of Vervaat's transform to continuous path with finite but arbitrary life-time: Given a continuous function $f$ on $[0,T]$ and $\tau(f)$ the first time it attains its minimum, define:
\begin{displaymath}
V(f)(t):=f(\tau(f)+t \mod T)+f(T)1_{\lbrace t+\tau(f)\geq T\rbrace}-f(\tau(f))
\end{displaymath}
From identity \eqref{Eq1} one gets:
\begin{equation}
\label{Eq2}
\begin{split}
\int_{0}^{+\infty}dT\,p_{T}(0,\lambda)V_{\ast}(\mathbb{P}^{T}_{0,\lambda})&=2\int_{-\infty}^{\lambda}
dy\,\mathbb{P}_{\lambda-y}^{T_{0}\wedge}\circ\mathbb{P}_{\lambda-y}^{T_{0}}\circ
\mathbb{P}_{0}^{T_{\lambda}}\\&=
2\left(\int_{0}^{+\infty}
dy\,\mathbb{P}_{y}^{T_{0}\wedge}\circ\mathbb{P}_{y}^{T_{0}}\right)\circ
\mathbb{P}_{0}^{T_{\lambda}}
\end{split}
\end{equation}
For bridges from $0$ to $0$ identity \eqref{Eq2} becomes:
\begin{displaymath}
\int_{0}^{+\infty}dT\,p_{T}(0,0)V_{\ast}(\mathbb{P}^{T}_{0,0})=
2\int_{0}^{+\infty}
dy\,\mathbb{P}_{y}^{T_{0}\wedge}\circ\mathbb{P}_{y}^{T_{0}}
\end{displaymath}
Let $Q^{T}_{0,0}$ be the law of positive Brownian excursion of length $T$ (bridge of Bessel $3$ from $0$ to $0$). According to Vervaat's result (\cite{Vervaat1979BridgeToExc}), $V_{\ast}(\mathbb{P}^{T}_{0,0})=Q^{T}_{0,0}$. Thus
\begin{displaymath}
2\int_{0}^{+\infty}
dy\,\mathbb{P}_{y}^{T_{0}\wedge}\circ\mathbb{P}_{y}^{T_{0}}=
\int_{0}^{+\infty}dT\,p_{T}(0,0)Q^{T}_{0,0}
\end{displaymath}
By injecting the above identity in \eqref{Eq2} one gets:
\begin{equation}
\label{Eq3}
\int_{0}^{+\infty}dT\,p_{T}(0,\lambda)V_{\ast}(\mathbb{P}^{T}_{0,\lambda})=
\left(\int_{0}^{+\infty}
ds\,p_{s}(0,0)Q^{s}_{0,0}\right)\circ
\mathbb{P}_{0}^{T_{\lambda}}
\end{equation}
By disintegrating in \eqref{Eq3} with respect to the life-time of paths, on gets that $V_{\ast}(\mathbb{P}^{T}_{0,\lambda})$ is a concatenation of an excursion and a first passage bridge.

Let $1_{t>0}f_{fp,\vert\lambda\vert}(t)\,dt$ be the distribution of the first hitting time of level $\lambda$ for a Brownian motion staring from $0$.
\begin{displaymath}
f_{fp,\vert\lambda\vert}(t)=\dfrac{\vert\lambda\vert}{\sqrt{2\pi t^{3}}}
\exp\left(-\dfrac{\lambda^{2}}{2t}\right)
\end{displaymath}
From \eqref{Eq3} follows that the distribution of the point of split $Z$ between the excursion and the first passage bridge in $V_{\ast}(\mathbb{P}^{1}_{0,\lambda})$ is:
\begin{displaymath}
\dfrac{p_{t}(0,0)f_{fp,\vert\lambda\vert}(1-t)}{p_{1}(0,\lambda)}\,1_{0<t<1}dt=
\dfrac{\vert\lambda\vert}{\sqrt{2\pi t(1-t)^{3}}}\exp\left(-\dfrac{\lambda^{2}t}{2(1-t)}\right)\,1_{0<t<1}dt
\end{displaymath}

\section{Semi-martingale decomposition of the Vervaat's transform of a bridge with positive endpoint}

Let $\lambda>0$. The transformed bridge $(V(B^{\lambda,br})_{t})_{0\leq t\leq 1}$ can be decomposed in a time-reversed first passage bridge from $\lambda$ to $0$ and a positive excursion above $\lambda$. Let $\widehat{Z}$ be the position of the split between the two. The density of the distribution of $\widehat{Z}$ on $(0,1)$ is given by:
\begin{displaymath}
f_{\widehat{Z}}(t)=\dfrac{\lambda}{\sqrt{2\pi (1-t)t^{3}}}\exp\left(-\dfrac{\lambda^{2}(1-t)}{2t}\right)\,1_{0<t<1}dt
\end{displaymath}
Let $(R_{t})_{t\geq 0}$ be a Bessel $3$ process starting from $0$. We will show that for any $t\in[0,1)$, the law of $(V(B^{\lambda,br})_{s})_{0\leq s\leq t}$ is absolutely continuous with respect to the law of $(R_{s})_{0\leq s\leq t}$, identify the corresponding density $D^{\lambda}_{t}$, and deduce by applying Girsanov's theorem the semi-martingale decomposition of $(V(B^{\lambda,br})_{t})_{0\leq t\leq 1}$.

For $x,y >0$, denote
\begin{displaymath}
\tilde{q}_{t}(x,y):=\dfrac{1}{xy\sqrt{2\pi t}}\left(
\exp\left(-\dfrac{(y-x)^{2}}{2t}\right)-\exp\left(-\dfrac{(y+x)^{2}}{2t}\right)
\right)
\end{displaymath}
$\tilde{q}_{t}(x,y) y^{2}\,dy$ is the semi-group of Bessel $3$. Let
\begin{displaymath}
\tilde{q}_{t}(0,y)=\lim_{x\rightarrow 0^{+}}\tilde{q}_{t}(x,y)=
\dfrac{2}{\sqrt{2\pi t^{3}}}\exp\left(-\dfrac{y^{2}}{2t}\right)=\dfrac{2}{y}f_{fp, y}(t)
\end{displaymath}
\begin{displaymath}
\tilde{q}_{t}(0,0)=
\dfrac{2}{\sqrt{2\pi t^{3}}}
\end{displaymath}
For $x,y\geq 0$, let $Q^{t}_{x,y}$ be the law of the bridge of Bessel $3$ from $x$ to $y$ of length $t$. The first passage bridge from $x$ to $0$ of length $t$ of the Brownian motion has the law $Q^{t}_{x,0}$ (\cite{BianeYor1988PrecisionsMeandre}). Let
\begin{displaymath}
\theta^{\lambda}_{t}:=\sup\lbrace s\in[0,t]\vert R_{s}\leq \lambda\rbrace
\end{displaymath}
If $R_{t}\leq\lambda$ then $\theta^{\lambda}_{t}=t$. The density $D^{\lambda}_{t}$ will be expressed as a deterministic function of $t, R_{t}$ and $\theta^{\lambda}_{t}$.

\begin{lem}
\label{Lem1}
On the event $R_{t}>\lambda$, the joint distribution of $(R_{T},\theta^{\lambda}_{T})$ is:
\begin{multline*}
\tilde{q}_{t}(0,y)\dfrac{f_{fp,y-\lambda}(t-s)f_{fp,\lambda}(s)}{f_{fp,y}(t)}\,1_{0<s<t}\,ds\,
1_{y>\lambda}y^{2}\,dy=\\
\dfrac{2y^{2}(y-\lambda)\lambda}{\pi\sqrt{(t-s)^{3}s^{3}}}
\exp\left(-\dfrac{(y-\lambda)^{2}}{2(t-s)}-\dfrac{\lambda^{2}}{2s}\right)\,1_{0<s<t}\,ds\,1_{y>\lambda}\,dy
\end{multline*}
Conditionally on $R_{t}>\lambda$, the value of $R_{t}$ and of $\theta^{\lambda}_{t}$, the paths $(R_{s})_{0\leq s\leq \theta^{\lambda}_{t}}$ and $(R_{\theta^{\lambda}_{t}-s}-\lambda)_{0\leq s\leq t-\theta^{\lambda}_{t}}$ are independent and follow the law $Q^{\theta^{\lambda}_{t}}_{0,\lambda}$ respectively $Q^{t-\theta^{\lambda}_{t}}_{0,R_{t}-\lambda}$.
\end{lem}
\begin{proof}
Let $y>\lambda$. Conditionally on $R_{t}=y$, $(R_{t-s})_{0\leq s\leq t}$ is a Brownian first passage bridge from $y$ to $0$ and $t-\theta^{\lambda}_{t}$ is the first time it hits $\lambda$. Thus conditionally on $R_{t}=y$, $t-\theta^{\lambda}_{t}$ is distributed according:
\begin{displaymath}
\dfrac{f_{fp,y-\lambda}(s)f_{fp,\lambda}(t-s)}{f_{fp,y}(t)}\,1_{0<s<t}\,ds
\end{displaymath}
Moreover conditionally on $R_{t}=y$ and on the value of $\theta^{\lambda}_{t}$, $(R_{t-s})_{0\leq s\leq t-\theta^{\lambda}_{t}}$ and $(R_{\theta^{\lambda}_{t}-s})_{0\leq s\leq \theta^{\lambda}_{t}}$ are two independent Brownian first passage bridges, from $y$ to $\lambda$ and from $\lambda$ to $0$.
\end{proof}

\begin{prop}
For any $t\in[0,1)$, the law of $(V(B^{\lambda,br})_{s})_{0\leq s\leq t}$ is absolutely continuous with respect to the law of $(R_{s})_{0\leq s\leq t}$. The corresponding density is:
\begin{equation*}
\begin{split}
D^{\lambda}_{t}=&\dfrac{\exp\left(\frac{\lambda^{2}}{2}\right)}{2R_{t}}\int_{t}^{1}\dfrac{ds}{\sqrt{2\pi(1-s)(s-t)}}\left(\exp\left(-\dfrac{(R_{t}-\lambda)^{2}}{2(s-t)}\right)-
\exp\left(-\dfrac{(R_{t}+\lambda)^{2}}{2(s-t)}\right)\right)\\+&
1_{R_{t}>\lambda}\dfrac{(1-\theta^{\lambda}_{t})(R_{t}-\lambda)}{\sqrt{(1-t)^{3}}R_{t}}\exp\left(\dfrac{\lambda^{2}}{2}\right)
\exp\left(-\dfrac{(R_{t}-\lambda)^{2}}{2(1-t)}\right)
\\:=&F^{\lambda}(t,R_{t},\theta^{\lambda}_{t})
\end{split}
\end{equation*}
\end{prop}
\begin{proof}
Observe that as as stochastic process, $(D^{\lambda}_{t})_{0\leq t<1}$ is continuous and in particular there is no discontinuity as $R_{t}$ crosses the level $\lambda$.

Let $t\in (0,1)$. We will decompose the density $D^{\lambda}_{t}$ as sum of two parts: 
$D^{\lambda}_{t}=D^{1,\lambda}_{t}+D^{2,\lambda}_{t}$, $D^{1,\lambda}_{t}$ accounting for the situation $\widehat{Z}>t$ and $D^{1,\lambda}_{t}$ for the situation $\widehat{Z}<t$. On the event $R_{t}<\lambda$, we will have $D^{\lambda}_{t}=D^{1,\lambda}_{t}$.

Conditionally on $\widehat{Z}>t$ and on the position of $V(B^{\lambda,br})_{t}$, the paths $(V(B^{\lambda,br})_{s})_{0\leq s\leq t}$ is a Bessel $3$ bridge from $0$ to $V(B^{\lambda,br})_{t}$, i.e. these are the same conditional laws as the laws of $(R_{s})_{0\leq s\leq t}$ conditioned on the value of $R_{t}$. Conditionally on $\widehat{Z}>t$ and on the value of $\widehat{Z}$, the distribution of $V(B^{\lambda,br})_{t}$ is:
\begin{displaymath}
\dfrac{\tilde{q}_{t}(0,y)\tilde{q}_{\widehat{Z}-t}(y,\lambda)}{\tilde{q}_{\widehat{Z}}(0,\lambda)}\,1_{y>0}\,y^{2}dy
\end{displaymath}
Let
\begin{multline*}
D^{1,\lambda}_{t}:=\int_{t}^{1}\dfrac{\tilde{q}_{s-t}(R_{t},\lambda)}{\tilde{q}_{s}(0,\lambda)}f_{\widehat{Z}}(s)\,ds=\\\dfrac{\exp\left(\frac{\lambda^{2}}{2}\right)}{2R_{t}}\int_{t}^{1}\dfrac{ds}{\sqrt{2\pi(1-s)(s-t)}}\left(\exp\left(-\dfrac{(R_{T}-\lambda)^{2}}{2(s-T)}\right)-
\exp\left(-\dfrac{(R_{t}+\lambda)^{2}}{2(s-t)}\right)\right)
\end{multline*}
Then for any measurable bounded functional $\Phi$ on paths:
\begin{displaymath}
\mathbb{E}\left[D^{1,\lambda}_{t}\Phi((R_{s})_{0\leq s\leq t})\right]=
\mathbb{E}\left[\Phi((V(B^{\lambda,br})_{s})_{0\leq s\leq t})1_{\widehat{Z}>T}\right]
\end{displaymath}

Next we consider the case $\widehat{Z}<t$. Conditionally on $\widehat{Z}<t$ and the position of $\widehat{Z}$ and $V(B^{\lambda,br})_{t}$, the paths $(V(B^{\lambda,br})_{s})_{0\leq s\leq \widehat{Z}}$ and 
$(V(B^{\lambda,br})_{\widehat{Z}+s}-\lambda)_{0\leq s\leq t-\widehat{Z}}$ are independent and follow the law $Q^{\widehat{Z}}_{0,\lambda}$ respectively $Q^{t-\widehat{Z}}_{0,V(B^{\lambda,br})_{t}-\lambda}$. These are the same conditional laws as in lemma \ref{Lem1}. On the event $\widehat{Z}<t$, the joint distribution of $(V(B^{\lambda,br})_{t},\widehat{Z})$ is:
\begin{displaymath}
f_{\widehat{Z}}(s)\dfrac{\tilde{q}_{t-s}(0,y-\lambda)\tilde{q}_{1-t}(y-\lambda,0)}{\tilde{q}_{1-s}(0,0)}\,1_{y>\lambda}\,(y-\lambda)^{2}dy\,1_{0<s<t}\,ds
\end{displaymath}
Let
\begin{equation*}
\begin{split}
D^{2,\lambda}_{t}=&1_{R_{t}>\lambda}
\dfrac{f_{\widehat{Z}}(\theta^{\lambda}_{t})\dfrac{\tilde{q}_{t-\theta^{\lambda}_{t}}(0,R_{t}-\lambda)\tilde{q}_{1-t}(R_{t}-\lambda,0)}{\tilde{q}_{1-\theta^{\lambda}_{t}}(0,0)}(R_{t}-\lambda)^{2}}
{\tilde{q}_{t}(0,R_{t})\dfrac{f_{fp,R_{t}-\lambda}(t-\theta^{\lambda}_{t})f_{fp,\lambda}(\theta^{\lambda}_{t})}{f_{fp,R_{t}}(t)}\,R_{t}^{2}}\\
=&1_{R_{t}>\lambda}\dfrac{(1-\theta^{\lambda}_{t})(R_{t}-\lambda)}{\sqrt{(1-t)^{3}}R_{T}}\exp\left(\dfrac{\lambda^{2}}{2}\right)
\exp\left(-\dfrac{(R_{t}-\lambda)^{2}}{2(1-T)}\right)
\end{split}
\end{equation*}
By construction, for any $\Phi$ measurable bounded function on $\mathbb{R}^{2}$:
\begin{displaymath}
\mathbb{E}\left[D^{2,\lambda}_{t}\Phi(R_{t},\theta^{\lambda}_{t})\right]=
\mathbb{E}\left[\Phi(V(B^{\lambda,br})_{t},\widehat{Z})1_{\widehat{Z}<t}\right]
\end{displaymath}
Because of the equality of conditional laws, for any measurable bounded functional $\Phi$ on paths:
\begin{displaymath}
\mathbb{E}\left[D^{2,\lambda}_{t}\Phi((R_{s})_{0\leq s\leq t})\right]=
\mathbb{E}\left[\Phi((V(B^{\lambda,br})_{s})_{0\leq s\leq t})1_{\widehat{Z}<t}\right]
\end{displaymath}
\end{proof}

\begin{lem}
\label{LemIntegrale}
For any $t\in(0,1)$ and $a\geq 0$:
\begin{displaymath}
\int_{t}^{1}\dfrac{ds}{\sqrt{(1-s)(s-t)}}\exp\left(-\dfrac{a}{s-t}\right)=\sqrt{\pi}
\int^{+\infty}_{\frac{a}{1-t}}e^{-u}\dfrac{du}{\sqrt{u}}
\end{displaymath}
\end{lem}

\begin{proof}
With the change of variables $z:=\dfrac{1-s}{1-t}$ we get
\begin{displaymath}
\int_{t}^{1}\dfrac{ds}{\sqrt{(1-s)(s-t)}}\exp\left(-\dfrac{a}{s-t}\right)=
\int_{0}^{1}\dfrac{dz}{\sqrt{z(1-z)}}\exp\left(-\dfrac{a}{(1-t)z}\right)
\end{displaymath}
For $x\geq 0$ let:
\begin{displaymath}
\varphi(x):=\int_{0}^{1}\dfrac{dz}{\sqrt{z(1-z)}}\exp\left(-\dfrac{x}{z}\right)
\end{displaymath}
With the change of variables $v=z^{-1}$ we get:
\begin{displaymath}
\varphi(x)=\int_{1}^{+\infty}\dfrac{dv}{v\sqrt{v-1}}e^{-xv}
\end{displaymath}
Differentiating with respect to $x$ we get:
\begin{equation*}
\begin{split}
\varphi'(x)=&-\int_{1}^{+\infty}\dfrac{dv}{\sqrt{v-1}}e^{-xv}=
-e^{-x}\int_{0}^{+\infty}\dfrac{dv}{\sqrt{v}}e^{-xv}\\=&
-\dfrac{e^{-x}}{\sqrt{x}}\int_{0}^{+\infty}\dfrac{dv}{\sqrt{v}}e^{-v}=
-\sqrt{\pi}\dfrac{e^{-x}}{\sqrt{x}}
\end{split}
\end{equation*}
Moreover $\varphi$ satisfies the border condition $\varphi(+\infty)=0$. Thus
\begin{displaymath}
\varphi(x)=\sqrt{\pi}\int_{x}^{+\infty}e^{-u}\dfrac{du}{\sqrt{u}}
\end{displaymath}
\end{proof}

Let
\begin{displaymath}
F^{1,\lambda}(t,y):=\dfrac{1}{2\sqrt{2}y}
\exp\left(\frac{\lambda^{2}}{2}\right)\int_{\frac{(y-\lambda)^{2}}{2(1-t)}}^{\frac{(y+\lambda)^{2}}{2(1-t)}} 
e^{-u}\dfrac{du}{\sqrt{u}}
\end{displaymath}
\begin{displaymath}
F^{2,\lambda}(t,y):=\dfrac{(y-\lambda)}{\sqrt{(1-t)^{3}}y}\exp\left(\dfrac{\lambda^{2}}{2}\right)
\exp\left(-\dfrac{(y-\lambda)^{2}}{2(1-t)}\right)
\end{displaymath}
According to lemma \ref{LemIntegrale}:
\begin{displaymath}
F^{\lambda}(t,y,\theta)=F^{1,\lambda}(t,y)+(1-\theta)(0\vee F^{2,\lambda}(t,y))
\end{displaymath}
$F^{2,\lambda}$ is $\mathcal{C}^{1}$. $F^{1,\lambda}$ and the partial derivative $\partial_{1}F^{1,\lambda}$ are continuous as functions of two variables. Yet $\partial_{2}F^{1,\lambda}(t,y)$ is not defined at $y=\lambda$:
\begin{displaymath}
\partial_{2}F^{1,\lambda}(t, \lambda^{+})-\partial_{2}F^{1,\lambda}(t, \lambda^{-})=
-\dfrac{1}{\sqrt{1-t}\lambda}\exp\left(\dfrac{\lambda^{2}}{2}\right)
\end{displaymath}
\begin{equation*}
\begin{split}
\partial_{2}F^{\lambda}(t, \lambda^{+},\theta)-\partial_{2}F^{\lambda}(t,&\lambda^{-},\theta)
\\=&\partial_{2}F^{1,\lambda}(t, \lambda^{+})-\partial_{2}F^{1,\lambda}(t, \lambda^{-})
+(1-\theta)\partial_{2}F^{2,\lambda}(t,\lambda)\\
=&\dfrac{(t-\theta)}{\sqrt{(1-t)^{3}}\lambda}\exp\left(\dfrac{\lambda^{2}}{2}\right)
\end{split}
\end{equation*}

For $t>0$ let: 
\begin{displaymath}
W_{t}:=R_{t}-\int_{0}^{t}\dfrac{ds}{R_{s}}
\end{displaymath}
$(W_{t})_{t\geq 0}$ is a standard Brownian motion starting from $0$, predictable with respect the filtration of $(R_{t})_{t\geq 0}$.

\begin{lem}
\label{LemMartDec}
For all $t\in[0,1)$:
\begin{displaymath}
D^{\lambda}_{t}=1+\int_{0}^{t}\partial_{2}F^{\lambda}(s,R_{s},\theta^{\lambda}_{s})\,dW_{s}
\end{displaymath}
\end{lem}
\begin{proof}
One can not just apply Ito's formula to $F^{\lambda}(t,R_{t},\theta^{\lambda}_{t})$ because 
$(\theta^{\lambda}_{t})_{t\geq 0}$ is a process that has jumps.

One can check that $F^{2,\lambda}$ and $F^{1,\lambda}$ outside $\lbrace y=\lambda\rbrace$ satisfy the PDE:
\begin{displaymath}
\dfrac{1}{2}\partial_{2,2}F(t,y)+\dfrac{1}{y}\partial_{2}F(t,y)+\partial_{1}F(t,y)=0
\end{displaymath}
Let $(L^{\lambda}_{t}(R))_{t\geq 0}$ be the local time at level $\lambda$ of $(R_{t})_{t\geq 0}$. Applying Ito-Tanaka's formula, and taking in account the discontinuity of partial derivatives $\partial_{2}$ at level $y=\lambda$, we get:
\begin{displaymath}
F^{1,\lambda}(t,R_{t})=1+\int_{0}^{t}\partial_{2}F^{1,\lambda}(s,R_{s})\,dW_{s}-
\dfrac{1}{\lambda}\exp\left(\dfrac{\lambda^{2}}{2}\right)\int_{0}^{t}\dfrac{1}{\sqrt{1-s}}
\,dL^{\lambda}_{s}(R)
\end{displaymath}
\begin{displaymath}
0\vee F^{2,\lambda}(t,R_{t})=\int_{0}^{t}1_{R_{s}>\lambda}
\partial_{2}F^{2,\lambda}(s,R_{s})\,dW_{s}+
\dfrac{1}{\lambda}\exp\left(\dfrac{\lambda^{2}}{2}\right)\int_{0}^{t}\dfrac{1}{\sqrt{(1-s)^{3}}}
\,dL^{\lambda}_{s}(R)
\end{displaymath}
$(1-\theta^{\lambda}_{t})$ is constant on the intervals of time where $0\vee F^{2,\lambda}(t,R_{t})$ is positive. From theorem $4.2$, section $VI.4$, \cite{RevuzYor1999BMGrundlehren}, follows that:
\begin{equation*}
\begin{split}
(1-\theta^{\lambda}_{t}&)(0\vee F^{2,\lambda}(t,R_{t}))\\
=&\int_{0}^{t}1_{R_{s}>\lambda}
(1-\theta^{\lambda}_{s})\partial_{2}F^{2,\lambda}(s,R_{s})\,dW_{s}+
\dfrac{1}{\lambda}\exp\left(\dfrac{\lambda^{2}}{2}\right)\int_{0}^{t}\dfrac{(1-\theta^{\lambda}_{s})}
{\sqrt{(1-s)^{3}}}\,dL^{\lambda}_{s}(R)\\
=&\int_{0}^{t}1_{R_{s}>\lambda}
(1-\theta^{\lambda}_{s})\partial_{2}F^{2,\lambda}(s,R_{s})\,dW_{s}+
\dfrac{1}{\lambda}\exp\left(\dfrac{\lambda^{2}}{2}\right)\int_{0}^{t}\dfrac{1}
{\sqrt{1-s}}\,dL^{\lambda}_{s}(R)
\end{split}
\end{equation*}
on the support of $dL^{\lambda}_{s}(R)$, $(1-\theta^{\lambda}_{s})$ being equal to $s$. Finally
\begin{multline*}
F^{1,\lambda}(t,R_{t})+(1-\theta^{\lambda}_{t})(0\vee F^{2,\lambda}(t,R_{t}))\\=
1+\int_{0}^{t}(\partial_{2}F^{1,\lambda}(s,R_{s})+(1-\theta^{\lambda}_{s})\partial_{2}
F^{2,\lambda}(s,R_{s}))\,dW_{s}
\end{multline*}
which finishes the proof.
\end{proof}

\begin{prop}
For $t\in(0,1)$, let:
\begin{displaymath}
\tilde{\theta}_{t}^{\lambda,br}:=\sup\lbrace s\in[0,t]\vert V(B^{\lambda, br})_{s}\leq \lambda\rbrace
\end{displaymath}
Then the process
\begin{displaymath}
\left(V(B^{\lambda, br})_{t}-\int_{0}^{t}\dfrac{ds}{V(B^{\lambda, br})_{s}}-\int_{0}^{t}\dfrac{\partial_{2}F^{\lambda}}{F^{\lambda}}
(s,V(B^{\lambda, br})_{s},\tilde{\theta}_{s}^{\lambda,br})\,ds\right)_{0\leq t\leq 1}
\end{displaymath}
is a standard Brownian motion.
\end{prop}
\begin{proof}
For $t\in[0,1)$, let
\begin{displaymath}
X_{t}:=V(B^{\lambda, br})_{t}-\int_{0}^{t}\dfrac{ds}{V(B^{\lambda, br})_{s}}
\end{displaymath}
The law of $(X_{s})_{0\leq s\leq t}$ is absolutely continuous with respect to the law of $(W_{s})_{0\leq s\leq t}$, with density $D^{\lambda}_{t}$. From lemma \ref{LemMartDec} follows that
\begin{displaymath}
[\log(D^{\lambda}),W]_{t}=
\int_{0}^{t}\dfrac{\partial_{2}F^{\lambda}}{F^{\lambda}}(s,R_{s},\theta^{\lambda}_{s})\,ds
\end{displaymath}
From Girsanov's theorem follows that the process:
\begin{displaymath}
X_{t}-\int_{0}^{t}\dfrac{\partial_{2}F^{\lambda}}{F^{\lambda}}
(s,V(B^{\lambda, br})_{s},\tilde{\theta}_{s}^{\lambda,br})\,ds
\end{displaymath}
is a Brownian motion.
\end{proof}

\section{Semi-martingale decomposition of the Vervaat's transform of a Brownian path}

In this section we consider $(V(B)_{t})_{0\leq t\leq 1}$ the Vervaat's transform of a Brownian motion on $[0,1]$. By definition, $V(B)_{1}=B_{1}$. A.s. there is $\varepsilon>0$ such that for all $t\in(0,\varepsilon)$, $V(B)_{t}>0$. Let
\begin{displaymath}
\widetilde{T}_{0}:=\inf\lbrace t\in (0,1]\vert V(B)_{t}=0\rbrace
\end{displaymath}
Then $\mathbb{P}(\widetilde{T}_{0}\leq 1)=\dfrac{1}{2}$ and more precisely
$\lbrace \widetilde{T}_{0}\leq 1\rbrace = \lbrace V(B)_{1}\leq 0\rbrace$. Conditionally on $\widetilde{T}_{0}\leq 1$, $\widetilde{T}_{0}$ follows the arcsine law $1_{0< t<1}\dfrac{dt}{\pi\sqrt{t(1-t)}}$. Conditionally on $\widetilde{T}_{0}\leq 1$ and on the value of $\widetilde{T}_{0}$, $(V(B)_{t})_{0\leq t\leq \widetilde{T}_{0}}$
has the law $Q^{\widetilde{T}_{0}}_{0,0}$ and is independent from $(V(B)_{t})_{\widetilde{T}_{0}\leq t\leq 1}$.
The joint law of  $(V(B)_{1}, \widetilde{T}_{0})$ on the event $\widetilde{T}_{0}\leq 1$ is:
\begin{displaymath}
\dfrac{1_{\lambda <0}\,d \lambda}{\sqrt{2\pi}}\exp\left(-\dfrac{\lambda^{2}}{2}\right)
\dfrac{\vert\lambda\vert}{\sqrt{2\pi t(1-t)^{3}}}\exp\left(-\dfrac{\lambda^{2}t}{2(1-t)}\right)\,1_{0<t<1}dt
\end{displaymath}
Thus the law of $V(B)_{1}$ conditionally on $\widetilde{T}_{0}$ is:
\begin{equation}
\label{CondDensEndpoint}
\dfrac{\vert\lambda\vert}{1-\widetilde{T}_{0}}
\exp\left(-\dfrac{\lambda^{2}}{2(1-\widetilde{T}_{0})}\right)1_{\lambda <0}\,d \lambda
\end{equation}

For the semi-martingale decomposition of $(V(B)_{t})_{0\leq t\leq 1}$ we will split the task in two: the decomposition of $(V(B)_{t})_{0\leq t\leq \widetilde{T}_{0}}$ and the decomposition of 
$(V(B)_{t})_{\widetilde{T}_{0}\leq t\leq 1}$. We will start with the latter. Let $(\widetilde{M}_{t})_{t\geq 0}$ be the process:
\begin{displaymath}
\widetilde{M}_{t}:=\min_{[0,t]}V(B)
\end{displaymath}

\begin{lem}
\label{LemEqQueue}
Conditionally on the value of $\widetilde{T}_{0}$,
\begin{displaymath}
\left(V(B)_{t}+\int_{\widetilde{T}_{0}}^{t}\dfrac{V(B)_{s}-\widetilde{M}_{s}}{1-s}\,ds\right)
_{\widetilde{T}_{0}\leq t\leq 1}
\end{displaymath}
is a Brownian motion
\end{lem}

\begin{proof}
The value of $\widetilde{T}_{0}$ is considered as fixed. Let $(B'_{t})_{t\geq 0}$ be a Brownian motion starting from $0$ and
\begin{displaymath}
M'_{t}:=\min_{[0,t]}B'
\end{displaymath}
For any $t\in [\widetilde{T}_{0},1)$, the law of $(V(B)_{s})_{\widetilde{T}_{0}\leq t\leq t}$ is absolutely continuous with respect to the law of $(B'_{s})_{0\leq s\leq t-\widetilde{T}_{0}}$. The corresponding density is:
\begin{equation}
\label{DensityReste}
\int_{-\infty}^{M'_{t-\widetilde{T}_{0}}}\dfrac{f_{fp,B'_{t-\widetilde{T}_{0}}-\lambda}(1-t)}
{f_{fp,\vert\lambda\vert}(1-\widetilde{T}_{0})}
\dfrac{\vert\lambda\vert}{1-\widetilde{T}_{0}}
\exp\left(-\dfrac{\lambda^{2}}{2(1-\widetilde{T}_{0})}\right)\,d\lambda
\end{equation}
In above expression we integrate with respect to the density \eqref{CondDensEndpoint} the function:
\begin{displaymath}
1_{\lambda<M'_{t-\widetilde{T}_{0}}}\dfrac{f_{fp,B'_{t-\widetilde{T}_{0}}-\lambda}(1-t)}
{f_{fp,\vert\lambda\vert}(1-\widetilde{T}_{0})}
\end{displaymath}
which is the density corresponding to a Brownian first passage bridge from $0$ to $\lambda$ of length $1-\widetilde{T}_{0}$. \eqref{DensityReste} rewrites as:
\begin{multline*}
\sqrt{\dfrac{1-\widetilde{T}_{0}}{(1-t)^{3}}}
\int_{-\infty}^{M'_{t-\widetilde{T}_{0}}}(B'_{t-\widetilde{T}_{0}}-\lambda)
\exp\left(-\dfrac{(B'_{t-\widetilde{T}_{0}}-\lambda)^{2}}{2(1-t)}\right)\,d\lambda\\
=\sqrt{\dfrac{1-\widetilde{T}_{0}}{1-t}}
\exp\left(-\dfrac{(B'_{t-\widetilde{T}_{0}}-M'_{t-\widetilde{T}_{0}})^{2}}{2(1-t)}\right)
\end{multline*}
Applying Girsanov's theorem, we get the result of the lemma.
\end{proof}

Next we will deal with the semi-martingale decomposition of $(V(B)_{t\wedge\widetilde{T}_{0}})_{0\leq t\leq 1}$.
As an auxiliary problem we will study first the semi-martingale decomposition of a process $(\xi_{t})_{t\geq 0}$ defined as follows: with probability $\frac{1}{2}$, $\xi$ is a Bessel $3$ process starting from $0$. For $t\in (0,1)$, with infinitesimal probability $\frac{dt}{2\pi\sqrt{t(1-t)}}$, $\xi$ is a positive excursion of length $t$, absorbed at $0$ after time $t$. For any $t\in(0,1)$, the law of $(V(B)_{s\wedge\widetilde{T}_{0}})_{0\leq s\leq s}$ is absolutely continuous with respect the law of $(\xi_{s})_{0\leq s\leq t}$. 

\begin{lem}
\label{LemXi}
Let
\begin{displaymath}
T^{\xi}_{0}:=\inf\lbrace t>0\vert \xi_{t}=0\rbrace
\end{displaymath}
Let
\begin{displaymath}
J_{t}(y):=\int_{t\leq s\leq 1}\dfrac{ds}{\pi\sqrt{s(1-s)}}\dfrac{\tilde{q}_{s-t}(0,y)}{\tilde{q}_{s}(0,0)}
\end{displaymath}
\begin{displaymath}
\dot{J}_{t}(y):=\int_{t\leq s\leq 1}\dfrac{ds}{\pi(s-t)\sqrt{s(1-s)}}\dfrac{\tilde{q}_{s-t}(0,y)}
{\tilde{q}_{s}(0,0)}
\end{displaymath}
The process
\begin{displaymath}
\label{StoppedBM}
(Y_{t})_{t\geq 0}:=\left(\xi_{t}-\int_{0}^{t\wedge T^{\xi}_{0}}\dfrac{ds}{\xi_{s}}+
\int_{0}^{t\wedge T^{\xi}_{0}\wedge 1}\dfrac{\xi_{s}\dot{J}_{s}(\xi_{s})}{1+J_{s}(\xi_{s})}
\,ds\right)_{t\geq 0}
\end{displaymath}
is a Brownian motion with respect the filtration of $\xi$, stopped at time $T^{\xi}_{0}$.
\end{lem}

\begin{proof}
We would like to emphasize that $T^{\xi}_{0}$ is a stopping time for $\xi$ but not for process $Y$.

Let $\varepsilon\in(0,1)$. We introduce $(B^{\varepsilon}_{t})_{t\geq 0}$ a Brownian motion with the starting point $B^{\varepsilon}_{0}$ having the same distribution as $\xi_{\varepsilon\wedge T^{\xi}_{0}}$. Let $\mu_{\varepsilon}$ be the density of this distribution on $(0,+\infty)$ (total mass $<1$).
\begin{displaymath}
\mu_{\varepsilon}(x)=\dfrac{\tilde{q}_{\varepsilon}(0,x)x^{2}}{2}
\left(1+J_{\varepsilon}(x)\right)
\end{displaymath}
Let $T^{\varepsilon}_{0}$ be the first time $B^{\varepsilon}$ hits $0$. For any $t\varepsilon$, the law of $(\xi_{s})_{0\leq s\leq t}$ is absolutely continuous with respect the law $(B^{\varepsilon}_{(s-\varepsilon)\wedge T^{\varepsilon}_{0}})_{\varepsilon\leq s\leq t}$. The corresponding density is:

\begin{equation*}
\begin{split}
\mathfrak{D}^{\varepsilon}_{t}=&1_{B^{\varepsilon}_{0}=0}+\dfrac{\tilde{q}_{\varepsilon}(0,B^{\varepsilon}_{0})
B^{\varepsilon 2}_{0}\tilde{q}_{T^{\varepsilon}_{0}}(0,B^{\varepsilon}_{0})}
{\tilde{q}_{T^{\varepsilon}_{0}+\varepsilon}(0,0)2\pi\sqrt{T^{\varepsilon}_{0}(1-T^{\varepsilon}_{0})}
\mu_{\varepsilon}(B^{\varepsilon}_{0})f_{fp,B^{\varepsilon}_{0}}(T^{\varepsilon}_{0})}\,
1_{T^{\varepsilon}_{0}\leq t-\varepsilon, B^{\varepsilon}_{0}>0}\\
+&\dfrac{1_{T^{\varepsilon}_{0}>t-\varepsilon}}{\mu_{\varepsilon}(B^{\varepsilon}_{0})B^{\varepsilon}_{0}B^{\varepsilon}_{t-\varepsilon}\tilde{q}_{t-\varepsilon}(B^{\varepsilon}_{0},B^{\varepsilon}_{t-\varepsilon})}\times
\dfrac{\tilde{q}_{\varepsilon}(0,B^{\varepsilon}_{0})B^{\varepsilon 2}_{0}
\tilde{q}_{t-\varepsilon}(B^{\varepsilon}_{0},B^{\varepsilon}_{t-\varepsilon})B^{\varepsilon 2}_{t-\varepsilon}}{2}\\
\times&\left(1+\int_{t}^{1}\dfrac{ds}{\pi\sqrt{s(1-s)}}\dfrac{\tilde{q}_{t-s}
(0,B^{\varepsilon}_{t-\varepsilon})}{\tilde{q}_{s}(0,0)}\right)\\
=&1_{B^{\varepsilon}_{0}=0}+\dfrac{\tilde{q}_{\varepsilon}(0,B^{\varepsilon}_{0})
B^{\varepsilon}_{0}}
{\mu_{\varepsilon}(B^{\varepsilon}_{0})\pi\sqrt{T^{\varepsilon}_{0}(1-T^{\varepsilon}_{0})}
\tilde{q}_{T^{\varepsilon}_{0}+\varepsilon}(0,0)}\,
1_{T^{\varepsilon}_{0}\leq t-\varepsilon, B^{\varepsilon}_{0}>0}\\
+&\dfrac{1_{T^{\varepsilon}_{0}>t-\varepsilon}\tilde{q}_{\varepsilon}(0,B^{\varepsilon}_{0})B^{\varepsilon}_{0}
B^{\varepsilon}_{t-\varepsilon}}{2\mu_{\varepsilon}(B^{\varepsilon}_{0})}\times
\left(1+J_{t}(B^{\varepsilon}_{t-\varepsilon})\right)
\end{split}
\end{equation*}
$(\mathfrak{D}^{\varepsilon}_{t})_{t\geq 0}$ seen as time dependent process is continuous. In particular there is no discontinuity at $T^{\varepsilon}_{0}$. This follows from that fact that as $y$ tends to $0$, the convolution kernel
\begin{displaymath}
\dfrac{y}{2}\tilde{q}_{u}(0,y)\,1_{u>0}\,du
\end{displaymath}
is an approximation to the identity. Since for $t\in(0,1)$
\begin{displaymath}
\dfrac{\partial J_{t}(y)}{\partial y}=-y\dot{J}_{t}(y)
\end{displaymath}
applying Girsanov's theorem we get that $(Y_{t})_{t\geq \varepsilon}$ is a continuous martingale relatively the filtration of $(\xi_{t})_{t\geq \varepsilon}$ with quadratic variation $(t-\varepsilon)\wedge
(T^{\xi}_{0}-\varepsilon)^{+}$. Since this holds for all $\varepsilon$ sufficiently small, this implies the lemma.
\end{proof}

We introduce the functionals $F(t,\gamma)$ and $\dot{F}(t,\gamma)$ where $t$ is a time and $\gamma$ a continuous path:
\begin{displaymath}
F(t,\gamma):=\dfrac{2}{\sqrt{2\pi}}\int_{0}^{+\infty}
F^{\lambda}(t,\gamma(t),\sup\lbrace s\in [0,t]\vert \gamma(s)\leq \lambda\rbrace)
\exp\left(-\dfrac{\lambda^{2}}{2}\right)\,d\lambda
\end{displaymath}
\begin{displaymath}
\dot{F}(t,\gamma):=\dfrac{2}{\sqrt{2\pi}}\int_{0}^{+\infty}
\partial_{2}F^{\lambda}(t,\gamma(t),\sup\lbrace s\in [0,t]\vert \gamma(s)\leq \lambda\rbrace)
\exp\left(-\dfrac{\lambda^{2}}{2}\right)\,d\lambda
\end{displaymath}
For any $t\in(0,1)$, the law of $(V(B)_{s\wedge\widetilde{T}_{0}})_{0\leq s\leq s}$ is absolutely continuous with respect the law of $(\xi_{s})_{0\leq s\leq t}$ with density
\begin{equation}
\label{FinalDensity}
D_{t}=1_{T^{\xi}_{0}\leq t}+\dfrac{F(t,\xi)+J_{t}(\xi_{t})}{1+J_{t}(\xi_{t})}\,1_{T^{\xi}_{0}> t}
\end{equation}

\begin{lem}
\label{LemBound}
There is a positive function $c(t)$ bounded on intervals of form $[0,1-\varepsilon]$, such that for all 
$\lambda>0$, $y>0$, $\theta\leq t\in [0,1)$:
\begin{displaymath}
F^{\lambda}(t,y,\theta)\leq c(t)\exp\left(\frac{\lambda^{2}}{2}\right)
\exp\left(-\dfrac{(y-\lambda)^{2}}{2(1-t)}\right)
\end{displaymath} 
\end{lem}
\begin{proof}
$F^{\lambda}(t,y,\theta)\leq F^{1,\lambda}(t,y)+1_{y>\lambda}F_{2}(t,y)$. For $y>\lambda$:
\begin{displaymath}
F_{2}(t,y)\leq \dfrac{1}{\sqrt{(1-t)^{3}}}\exp\left(\frac{\lambda^{2}}{2}\right)
\exp\left(-\dfrac{(y-\lambda)^{2}}{2(1-t)}\right)
\end{displaymath}
For $F^{1,\lambda}(t,y)$ we have the inequality
\begin{equation*}
\begin{split}
F^{1,\lambda}(t,y)\leq&\dfrac{1}{2\sqrt{2}y}\exp\left(\frac{\lambda^{2}}{2}\right)
\exp\left(-\dfrac{(y-\lambda)^{2}}{2(1-t)}\right)
\int_{\frac{(y-\lambda)^{2}}{2(1-t)}}^{\frac{(y+\lambda)^{2}}{2(1-t)}}\dfrac{du}{\sqrt{u}}\\
=&\dfrac{2}{\sqrt{(1-t)}}\dfrac{\min(y,\lambda)}{y}
\exp\left(\frac{\lambda^{2}}{2}\right)
\exp\left(-\dfrac{(y-\lambda)^{2}}{2(1-t)}\right)
\end{split}
\end{equation*}
\end{proof}

\begin{lem}
\label{LemQuadVar}
For $t\in[0,1)$:
\begin{displaymath}
[F(\cdot ,\xi),\xi]_{t}=\int_{0}^{t\wedge T^{\xi}_{0}}\dot{F}(s,\xi)\,ds
\end{displaymath}
\end{lem}

\begin{proof}
It clear that the quadratic variation $[F(\cdot ,\xi),\xi]_{t}$ does not increase for $t\geq T^{\xi}_{0}$. We need only to show that for a Bessel $3$ process $(R_{t})_{t\geq 0}$
\begin{equation}
\label{EqQuadVar}
[F(\cdot ,R),R]_{t}=\int_{0}^{t}\dot{F}(s,R)\,ds
\end{equation}
Indeed, given any $T\in (0,1)$ and $t\in [0,T)$, the law of $(\xi_{s})_{0\leq s\leq t}$ on the event $T^{\xi}_{0}>T$ is absolutely continuous with respect the law of $(R_{s})_{0\leq s\leq t}$.

For any $\lambda>0$ $(F^{\lambda}(t,R_{t},\theta^{\lambda}_{t})_{0\leq t<1}$ is a positive martingale with mean $1$. Applying Fubini's theorem, we get that $(F(t,R))_{0\leq t<1}$ is a positive martingale with mean $1$. Let 
$(W_{t})_{t\geq 0}$ be the Brownian motion martingale part of $(R_{t})_{t\geq 0}$. To prove \eqref{EqQuadVar} we need only to show that the process
\begin{equation}
\label{ShowTrueMart}
\left(F(t,R)W_{t}-\int_{0}^{t}\dot{F}(s,R)\,ds\right)_{0\leq t<1}
\end{equation}
is a (true) martingale. Lemma \ref{LemMartDec} ensures that for any $\lambda>0$ the process
\begin{equation}
\label{EqMartLoc}
\left(F^{\lambda}(t,R_{t},\theta^{\lambda}_{t})W_{t}-\int_{0}^{t}
\partial_{2}F^{\lambda}(s,R_{s},\theta^{\lambda}_{s})\,ds\right)_{0\leq t<1}
\end{equation}
is a local martingale. Moreover from the bound of lemma \ref{LemBound} follows that $(F^{\lambda}(t,R_{t},\theta^{\lambda}_{t}))_{0\leq t<1}$ is a square integrable martingale and
\begin{equation*}
\begin{split}
\mathbb{E}\left[\int_{0}^{t}
\partial_{2}F^{\lambda}(s,R_{s},\theta^{\lambda}_{s})^{2}\,ds\right]=&
\mathbb{E}\left[F^{\lambda}(t,R_{t},\theta^{\lambda}_{t})^{2}\right]\\\leq&
c(t)^{2}\exp\left(\lambda^{2}\right)
\mathbb{E}\left[\exp\left(-\dfrac{(R_{t}-\lambda)^{2}}{(1-t)}\right)\right]
\end{split}
\end{equation*}
Thus
\begin{multline*}
\mathbb{E}\left[\left\vert
F^{\lambda}(t,R_{t},\theta^{\lambda}_{t})W_{t}-\int_{0}^{t}
\partial_{2}F^{\lambda}(s,R_{s},\theta^{\lambda}_{s})\,ds
\right\vert\right]\leq
2\sqrt{t}\mathbb{E}\left[
\int_{0}^{t}\partial_{2}F^{\lambda}(s,R_{s},\theta^{\lambda}_{s})^{2}\,ds
\right]^{\frac{1}{2}}\\
\leq 2\sqrt{t}c(t)\exp\left(\dfrac{\lambda^{2}}{2}\right)
\mathbb{E}\left[\exp\left(-\dfrac{(R_{t}-\lambda)^{2}}{(1-t)}\right)\right]^{\frac{1}{2}}
\end{multline*}
It follows that for any $\lambda >0$, the local martingale \eqref{EqMartLoc} is a true martingale and the expectation of its absolute value is integrable with respect $\dfrac{2}{\sqrt{2\pi}}\exp\left(-\dfrac{\lambda^{2}}{2}\right)\,1_{\lambda >0}\,d\lambda$. By Fubini's theorem, it follows that \eqref{ShowTrueMart} is a true martingale.
\end{proof}

\begin{prop}
The process
\begin{equation*}
\begin{split}
\Bigg(V(B)_{t}-&\int_{0}^{t\wedge\widetilde{T}_{0}}\dfrac{ds}{V(B)_{s}}+
\int_{0}^{t\wedge\widetilde{T}_{0}}\dfrac{\dot{F}(s,V(B))+
V(B)_{s}\dot{J}_{s}(V(B)_{s})}{F(s,V(B))+J_{s}(V(B)_{s})}\,ds
\\+&\int_{\widetilde{T}_{0}\leq s\leq t}\dfrac{V(B)_{s}-\widetilde{M}_{s}}{1-s}\,ds\Bigg)_{0\leq t\leq 1}
\end{split}
\end{equation*}
is a Brownian motion.
\end{prop}
\begin{proof}
The density process $(D_{t})_{0\leq t\leq 1}$ given by \eqref{FinalDensity} is time-continuous. In particular it follows from lemma \ref{LemBound} that on the event $T^{\xi}_{0}<1$, as $t$ converges to $T^{\xi}_{0}$ from below and $\xi_{t}$ converges to $0$, $F(t,\xi)$ remains bounded. Besides $J_{t}(\xi_{t})$ tends to $+\infty$ at $T^{\xi}_{0}$. Hence
\begin{displaymath}
\lim_{t\rightarrow T^{\xi}_{0}}\dfrac{F(t,\xi)+J_{t}(\xi_{t})}{1+J_{t}(\xi_{t})}=1
\end{displaymath}
and $D_{t}$ is continuous as $T^{\xi}_{0}$.

Using the semi-martingale decomposition of $(\xi_{t})_{t\geq 0}$ given by lemma \ref{LemXi}, applying the Girsanov's theorem together with lemma \ref{LemQuadVar} we get that the process
\begin{displaymath}
\left(V(B)_{t\wedge\widetilde{T}_{0}}-\int_{0}^{t\wedge\widetilde{T}_{0}}\dfrac{ds}{V(B)_{s}}+
\int_{0}^{t\wedge\widetilde{T}_{0}}\dfrac{\dot{F}(s,V(B))+
V(B)_{s}\dot{J}_{s}(V(B)_{s})}{F(s,V(B))+J_{s}(V(B)_{s})}\,ds\right)_{t\geq 0}
\end{displaymath}
is a martingale with quadratic variation $t\wedge\widetilde{T}_{0}$. Lemma \ref{LemEqQueue} describes the semi-martingale decomposition of $V(B)$ after the stopping time $\widetilde{T}_{0}$.
\end{proof}

\bibliographystyle{plain}
\bibliography{tituslupu}
\end{document}